\documentclass[11pt,reqno]{article}

\usepackage[usenames]{color}
\usepackage{amssymb}
\usepackage{amsmath}
\usepackage{amsthm}
\usepackage{amsfonts}
\usepackage{amscd}
\usepackage{graphicx}

\usepackage[colorlinks=true,
linkcolor=webgreen,
filecolor=webbrown,
citecolor=webgreen]{hyperref}

\definecolor{webgreen}{rgb}{0,.5,0}
\definecolor{webbrown}{rgb}{.6,0,0}

\usepackage{color}
\usepackage{fullpage}
\usepackage{float}

\usepackage{psfig}
\usepackage{graphics}
\usepackage{latexsym}
\usepackage{epsf}
\usepackage{breakurl}

\begin{document}

\begin{center}
\epsfxsize=4in
% \leavevmode\epsffile{logo129.eps}
\end{center}

\newtheorem{theorem}{Theorem}           
\newtheorem{lemma}{Lemma}               
\newtheorem{corollary}{Corollary}
\newtheorem{remark}{Remark}

\begin{center}
\vskip 1cm{\LARGE \bf Explicit Formulas For Generalized Polylogarithmic Integrals, Euler Sums, And BBP-Type Series}
\vskip 1cm {\large
Ali Olaikhan\\
Department of physics\\
University of Thi-Qar, Iraq\\
\href{mailto:alishathri@yahoo.com}{\tt alishathri@yahoo.com} \\}{February 22, 2025}
\end{center} 

\vskip .2in
\begin{abstract}
	This paper explores closed-form expressions for some polylogarithm integrals
	with integrands containing five parameters. These closed form expressions
	are given in terms of the Lerch transcendent
	function, which reduces, in some cases, to well-known special functions, such as the Riemann zeta, Dirichlet eta, Dirichlet beta, the Hurwitz zeta, and the polygamma functions. Related  Euler sums and BBP-type series will also be discussed.
\end{abstract}


 \begin{center}
 	Subject class: 33B30, 65B10, 40C10, 40G10, 11M06
 \end{center}

\section{Introduction and preliminaries}

In this paper we derive closed-form expressions for the following two generalized integrals (see Theorems \ref{infinity integral1} and \ref{infinity integral2}):
\begin{equation}
	\int_0^\infty\frac{\ln^{q-1}(x)\operatorname{Li}_p(bx^{n})}{x(1-ax)}\mathrm{d} x\quad\text{and}\quad  \int_0^\infty\frac{\ln^{q-1}(x)\operatorname{Li}_p(bx^{2n})}{1-ax^2}\mathrm{d} x,\label{infinity integrals}
\end{equation}
where $p, q,n$ are positive integers, and $a, b = \pm 1$, with the restriction that $q \neq 1$ when $a=1$, as these integrals would otherwise diverge. For the case $b=1$, the Cauchy principal value ($\operatorname{P.V.}$) is considered for these integrals, as the polylogarithm function is defined by \cite[p.\,35]{ali}:
\begin{equation}
	\operatorname{Li}_p(x)=\frac{{(-1)}^{p-1}}{(p-1)!}\int_0^1\frac{x\ln^{p-1}(y)}{1-xy}\mathrm{d} y,\quad x\leqslant1.\label{polylog}
\end{equation}

Using integral manipulation and the polylogarithmic identity given in Lemma \ref{Li(bx)+Li(b/x)}, we can derive from the first integral in (\ref{infinity integrals}) the following integral (see Theorems \ref{from zero to one integral1}):
\begin{equation}
	\int_0^1\frac{\ln^{q-1}(x)\operatorname{Li}_p(bx^{n})}{1-ax}\mathrm{d} x,\label{integral1}
\end{equation}
where $p+q$ is an odd integer, $a, b = \pm 1$, and $n$ is a non-zero integer. 

Similarly, from the second integral in (\ref{infinity integrals}) we can extract the following integral (see Theorem \ref{from zero to one integral2}):	 
\begin{equation}
	\int_0^1\frac{\ln^{q-1}(x)\operatorname{Li}_p(bx^{2n})}{1-ax^2}\mathrm{d} x,\label{integral2}
\end{equation}	 	 
where $a={(-1)}^{p+q-1}$, $b = \pm 1$, and $n$ is a non-zero integer.

The integral in (\ref{integral1}) is related to the Euler sum (see Theorem \ref{thm of H_n}):
\begin{equation}
	\sum_{k=1}^\infty\frac{a^k \mathcal{H}_{nk}^{(p)}(b)}{k^q},\label{Euler sums1}
\end{equation}
as well as to the BBP-type series (see Theorem \ref{bbp sum1}):
\begin{equation}
	\sum_{k=1}^\infty {(a^n)}^{k}\, \mathcal{H}_k^{(p)}(b)\sum_{j=1}^{n}\frac{a^{j-1}}{{(nk+j)}^q},\label{BBP type1}
\end{equation}
where $\mathcal{H}_{k}^{(p)}(b)=\sum_{m=1}^{k}\frac{b^{m-1}}{m^p}$, noting that the case $b=1$ gives the $k$th generalized harmonic numbers of order $p$, denoted by $H_{k}^{(p)}$, while the case $b=-1$ gives the $k$th generalized skew-harmonic numbers of order $p$, denoted by $\overline{H}_{k}^{(p)}$. 

The integral in (\ref{integral2}) is related to the Euler sum (see Theorem \ref{thm of O_n}):
\begin{equation}
	\sum_{k=1}^\infty\frac{a^k \mathcal{O}_{nk}^{(p)}(b)}{k^q},\label{Euler sums2}
\end{equation}
as well as to the BBP-type series (see Theorem \ref{bbp sum2}):	
\begin{equation}
	\sum_{k=1}^\infty {(a^n)}^{k}\, \mathcal{H}_k^{(p)}(b)\sum_{j=1}^{n}\frac{a^{j-1}}{{(2nk+2j-1)}^q},\label{BBP type2}
\end{equation}
where $\mathcal{O}_{k}^{(p)}(b)=\sum_{m=1}^{k}\frac{b^{m-1}}{(2m-1)^p}$, noting that the case $b=1$ gives the $k$th generalized odd harmonic numbers of order $p$, denoted by $O_{k}^{(p)}$, while the case $b=-1$ gives the $k$th generalized odd skew-harmonic numbers of order $p$, denoted by $\overline{O}_{k}^{(p)}$. 

The cases where \( n = 1 \) and \( a, b = \pm 1 \) in (\ref{Euler sums1}) were first introduced by Flajolet and Salvy \cite{flaj} using contour integration methods.  

The case \((p, q, a, b) = (1, 1, -1, 1)\) in (\ref{BBP type2}), where \(n\) is an odd positive integer, was derived by Batir \cite[p.\,19]{batir}:
\[\sum_{k=1}^\infty {(-1)}^k H_k \sum_{j=1}^n \frac{{(-1)}^{j-1}}{2nk + 2j - 1} 
= nG - \frac{(n+1)\pi}{4} \ln(2) 
- \frac{\pi}{2} \sum_{k=0}^{\frac{n-3}{2}} \ln\left(1 + \sqrt{\frac{1 - \varphi_k}{2}}\right).\]
Here, \(G\) is the Catalan constant and \(\varphi_k = \cos\left(\frac{(2k + 1)\pi}{n}\right)\).

It is important to note that the generalization presented in Theorem \ref{bbp sum2} does not produce Batir's result because the case $p=b=1$ is undefined. However, apart from this exception, the generalization holds for all other cases.

The BBP-type formula takes the general form:
\[\sum_{k=0}^\infty \frac{1}{p^k}\sum_{j=1}^{n}\frac{a_j}{{(kn+j)}^q},\]
where $p,q,n$ are positive integers, and $(a_1,a_2,..., a_n)\in\mathbb{Z}^{n}$. This type of series is named after Bailey, Borwein, and Plouffe, who first discovered the BBP-type formula for the constant $\pi$ \cite[p. 903]{Bailey}:
\[\pi=\sum_{k=0}^\infty\frac{1}{{16}^{k}}\left(\frac{4}{8k+1}-\frac{2}{8k+4}-\frac{1}{8k+5}-\frac{1}{8k+6}\right).\]	

The closed-form expressions for the generalizations derived in this paper are expressed in terms of the Lerch transcendent function, which is defined by
\begin{equation}
	\Phi(c,q,\alpha) = \sum_{n=0}^\infty \frac{c^n}{{(n+\alpha)}^q}, \label{series form}
\end{equation}
where $\alpha\ne0,-1,-2,..$, $q\geqslant1$, and $c=\pm1$ , with $q\ne1$ when $c=1$, as the sum woud otherwise diverge.

The case $c=1$ in (\ref{series form}) is the definition of the Hurwitz zeta function:
\[\Phi(1,q,\alpha)=\sum_{n=0}^\infty \frac{1}{{(n+\alpha)}^q}=\zeta(q,\alpha).\]
Employing the summation identity \cite[p. 4]{ali}:
\begin{equation}
	\sum_{n=0}^\infty a_{2n}=\frac12\sum_{n=0}^\infty(1+{(-1)}^{n}) a_n,\label{a_(2k) to a_k}
\end{equation}
we can find the case $c=-1$ in (\ref{series form}):
\begin{align*}
	\Phi(-1,q,\alpha)&=\sum_{n=0}^\infty \frac{(-1)^n}{{(n+\alpha)}^q}\\
	&=2\sum_{n=0}^\infty \frac{1}{{(2n+\alpha)}^q}-\sum_{n=0}^\infty \frac{1}{{(n+\alpha)}^q}\\
	&=2^{1-q} \sum_{n=0}^\infty \frac{1}{{(n+\frac{\alpha}{2})}^q}-\sum_{n=0}^\infty \frac{1}{{(n+\alpha)}^q}\\
	&=2^{1-q}\zeta\left(q,\frac{\alpha}{2}\right)-\zeta(q,\alpha).
\end{align*}    

Other special functions that are related to the Lerch transcendent function are:\\
The polylogarithm function
\[\operatorname{Li}_s(z) = \sum_{n=1}^\infty \frac{z^n}{n^s}=\sum_{n=0}^\infty \frac{z^{n+1}}{{(n+1)}^s} = z \, \Phi(z, s, 1),\]
the Riemann zeta function
\[\zeta(s) = \sum_{n=1}^\infty \frac{1}{n^s} =\sum_{n=0}^\infty \frac{1}{{(n+1)}^s}= \Phi(1, s, 1),\]
the Dirichlet eta function
\[\eta(s) = \sum_{n=1}^\infty \frac{{(-1)}^{n-1}}{n^s} = \sum_{n=0}^\infty \frac{{(-1)}^{n}}{{(n+1)}^s} = \Phi(-1, s, 1),\]
the Dirichlet lambda function
\[\lambda(s) = \sum_{n=0}^\infty \frac{1}{{(2n+1)}^s}=2^{-q}\sum_{n=0}^\infty \frac{1}{{(n+\frac12)}^s}= 2^{-s} \, \Phi(1, s, \tfrac{1}{2}),\]
and the Dirichlet beta function
\[\beta(s) = \sum_{n=0}^\infty \frac{{(-1)}^{n}}{{(2n+1)}^s}=2^{-q}\sum_{n=0}^\infty \frac{{(-1)}^{n}}{{(n+\frac12)}^s} = 2^{-s} \, \Phi(-1, s, \tfrac{1}{2}).\]
At odd values, the Dirichlet beta function reduces to powers of $\pi$ \cite[p.32]{ali}:
\[\beta(2s+1)=\frac{|E_{2s}|}{2(2s)!}\left(\frac{\pi}{2}\right)^{2s+1}.\]

Some other functions can be represented by the Lerch transcendent function, but they will not be discussed here as  they are unrelated to this work.

Most of the identities derived here are believed to be new in the literature, and they provide results in closed form that cannot be obtained by mathematical software such as \textit{Mathematica} or \textit{Maple}.

%%%%%%%%%%%%%%%%%%%%%%%%%%%%%%%%%%%%%%%%%%%%%%
\section{Lemmas}
\begin{lemma}\label{Li(bx)+Li(b/x)}
	Let $p\in\mathbb{Z}^{+}$, $c=\pm 1$ and $x\in (0,1)$. Then
	\[\operatorname{Li}_{p}(cx)+{(-1)}^{p}\operatorname{Li}_{p}\left(\frac{c}{x}\right)=2c\sum_{k=0}^{\lfloor{\frac{p}{2}}\rfloor} \frac{\ln^{p-2k}(x)}{(p-2k)!}\,\Phi(c,2k,1),\]
	where the real part of $\operatorname{Li}_p(\frac{c}{x})$ is considered when $c=1$.
\end{lemma}
\begin{proof}
	The case $c=-1$ is given by Lewin \cite[p.192]{Lewin}, while the case $c=1$ can be proved by repeatedly integrating the dilogarithm identity:
	\[\operatorname{Li}_2(x)+\mathfrak{R}\operatorname{Li}_2\left(\frac{1}{x}\right)=2\zeta(2)-\frac12\ln^2(x),\quad x\in(0,1),\]
	which is derived by differentiating $\operatorname{Li}_2(1/x)$ and then integrating back. 
\end{proof}
\begin{remark}\label{first remark}
	For the case $c=1$ and $k=0$ in Lemma \ref{Li(bx)+Li(b/x)}, view $\Phi(1,0,1)$ as $\zeta(0)=-1/2$ due to the analytic continuation of the Riemann zeta function, and consider the real part of $\operatorname{Li}_{p}(1/x)$. \textit{Mathematica} does not directly evaluate $\mathtt{LerchPhi[1, 0, 1]}$ as $\mathtt{-1/2}$, but it provides this value if $\mathtt{LerchPhi[1, 0, 1]}$ is expressed as a limit: 
	\[\mathtt{Limit[LerchPhi[1, k, 1], {k-> 0}] = -1/2}.\]
\end{remark}
%%%%%%%%%%%%%%%%%%%%%%%%%%%%%%%%%%%%%%%%%%%%%%%%%%%%%%%%%%%%%%%
\begin{lemma}
	\label{x^(-s)}
	Let $c=\pm 1$ and $s\in(0,1)$. Then we have
	\[\int_0^\infty\frac{x^{s-1}}{1-cx}\mathrm{d} x=-2c\sum_{n=0}^\infty \Phi(c,2n,1) s^{2n-1}.\] 
	The integral must be understood as the Cauchy principal value when $c=1$.
\end{lemma}
\begin{proof}
	From Cornel's second book \cite[p.\,258]{book2}, we have that
	\begin{equation}\label{grad}
		\int_0^\infty\frac{x^{s-1}}{x+c}\mathrm{d} x=
		\begin{cases}
			& -(-c)^{s-1}\pi\cot(\pi s)\quad c<0;\\
			&  c^{s-1}\pi\csc(\pi s)\,\,\,\quad\qquad c>0.
		\end{cases}
	\end{equation}
	The lemma follows from considering the cases $c=\pm1$, then expanding  the cotangent and cosecant functions in series:
	\[	\pi \cot(\pi x)=-2\sum_{n=0}^\infty \zeta(2n) x^{2n-1},\quad\pi\csc(\pi x)=2\sum_{n=0}^\infty \eta(2n) x^{2n-1},\qquad |x|<1,\]
	where the series expansion of $\cot(\pi x)$ is given in \cite[p.\,29]{ali}, and the series expansion of $\csc(\pi x)$ follows from using the trigonometric identity $\csc(x)=\cot\left(\frac{x}{2}\right)-\cot(x)$.
\end{proof} 
%%%%%%%%%%%%%%%%%%%%%%%%%%%%%%%%%%%%%%%%%%%%%%%%%%%
\begin{lemma}\label{lerch integral}
	Let $q\in\mathbb{Z}^{+}$, $s,r>0$, and $c=\pm1$, with $q\neq1$ when $c=1$. Then
	\[\int_0^1\frac{x^{s-1}\ln^{q-1}(x)}{1-cx^r}\mathrm{d} x=\frac{{(-1)}^{q-1}(q-1)!}{r^{q}}\,\Phi\left(c,q,\tfrac{s}{r}\right).\]
\end{lemma}
\begin{proof}
	This integral is derived by expanding $\frac{1}{1-c{x}^{r}}$ in a Taylor series, rearranging the order of summation and integration, which is justified by the uniform convergence theorem, then using the result \cite[p.\, 13]{ali}:
	\begin{equation}
		\int_0^1 x^{n-1} \ln^{q-1}(x)\mathrm{d} x=\frac{(-1)^{q-1}(q-1)!}{n^q},\label{classical}
	\end{equation}
	we obtain
	\[\int_0^1\frac{x^{s-1}\ln^{q-1}(x)}{1-cx^r}\mathrm{d} x=\sum_{n=0}^\infty c^n\int_0^1 x^{s+rn-1}\ln^{q-1}(x)\mathrm{d} x\]
	\[=(-1)^{q-1}(q-1)!\sum_{n=0}^\infty\frac{c^n}{{(rn+s)}^{q}}=\frac{(-1)^{q-1}(q-1)!}{r^{q}}\sum_{n=0}^\infty\frac{c^n}{{(n+\tfrac{s}{r})}^{q}}.\]
	The proof follows from (\ref{series form}).
\end{proof}

%%%%%%%%%%%%%%%%%%%%%%%%%%%%%%%%%%%%%%%%%%%%%%%%%%%%
\begin{lemma} \label{theta integral}
	Let $q\in\mathbb{Z}^+$, $c=\pm1$, and $\frac{s}{r}\in(0,1)$. Then
	\[\int_0^\infty\frac{x^{s-1}\ln^{q-1}(x)}{1-cx^r}\mathrm{d} x=\frac{{(-1)}^{q-1}(q-1)!}{r^{q}}\,\Theta(c,q,s,r),\]
	where
	\begin{equation}
		\Theta(c,q,s,r)=\Phi\left(c,q,\tfrac{s}{r}\right)+c{(-1)}^{q}\Phi\left(c,q,\tfrac{r-s}{r}\right),\label{theta}
	\end{equation}
	from which it follows that
	\[\Theta(c,q,1,2)=(1+c{(-1)}^{q})\,\Phi(c,q,\tfrac12).\]
\end{lemma}
\begin{proof}
	We begin with breaking up the integral at $x=1$, then letting $x\to1/x$ in the second integral, we find
	\[\int_0^\infty\frac{x^{s-1}\ln^{q-1}(x)}{1-cx^r}\mathrm{d} x=\int_0^1\frac{x^{s-1}\ln^{q-1}(x)}{1-cx^r}\mathrm{d} x+c{(-1)}^{q}\int_0^1\frac{x^{r-s-1}\ln^{q-1}(x)}{1-cx^r}\mathrm{d} x.\]
	The proof follows from Lemma \ref{lerch integral}.
\end{proof}

Additionally, by taking $(q,c)=(1,1)$ in this Lemma, considering the principal value of the integral, and then using the identity (\ref{grad}), the case $c=1$, we find
\[\Theta(1,1,s,r)=\operatorname{P.V.}\int_0^\infty\frac{r\,x^{s-1}}{1-x^r}\mathrm{d} x=\operatorname{P.V.}\int_0^\infty\frac{x^{s/r-1}}{1-x}\mathrm{d} x=\pi\cot\left(\tfrac{\pi s}{r}\right).\] 

%%%%%%%%%%%%%%%%%%%%%%%%%%%%%%%%%%%%%%%%%%%%%%%%%%%%%%%%%
\begin{lemma} \label{tough integral}
	Let $q,n\in\mathbb{Z}^{+}$, $y\in(0,1)$, $a,b=\pm1$, with $q\neq1$ when $a=1$. Then
	\[\int_0^\infty \frac{\ln^{q-1}(x)}{x}\left(\frac{1}{1-ax}-\frac{1}{1-byx^{n}}\right)\mathrm{d} x\]
	\[=-a(1+{(-1)}^{q})(q-1)!\,\Phi(a,q,1)
	+\frac{2b{(-1)}^{q} (q-1)!}{n^{q}}\sum_{j=0}^{\lfloor\frac{q}{2}\rfloor}\frac{\ln^{q-2j}(y)}{(q-2j)!}\,\Phi(b,2j,1).\]
\end{lemma}
\begin{proof}
	First, note that we cannot split the integrand, as both resulting integrals would diverge. To evaluate this integral, we first apply differentiation under the integral sign, which is justified by the dominated convergence theorem, that is
	\[\int_0^\infty \frac{\ln^{q-1}(x)}{x}\left(\frac{1}{1-ax}-\frac{1}{1-byx^{n}}\right)\mathrm{d} x
	=\lim_{s\to0}\frac{\partial^{q-1}}{\partial s^{q-1}}\left(\int_0^\infty\frac{x^{s-1}}{1-ax}\mathrm{d} x-\int_0^\infty\frac{x^{s-1}}{1-byx^n}\mathrm{d} x\right).\]
	Then by exploiting the identity (\ref{a_(2k) to a_k}) in Lemma \ref{x^(-s)}, we find for the first integral
	\begin{equation}
		\int_0^\infty\frac{x^{s-1}}{1-ax}\mathrm{d} x=-a\sum_{k=0}^\infty(1+{(-1)}^{k})\,\Phi(a,k,1)s^{k-1}.\label{sum1}
	\end{equation}
	For the second integral, substituting $yx^n=t$ and then making use of the series form given in Lemma \ref{x^(-s)}, we get
	\[\int_0^\infty\frac{x^{s-1}}{1-byx^n}\mathrm{d} x=-2by^{-\frac{s}{n}}\sum_{k=0}^\infty n^{-2k}\,\Phi(b,2k,1)s^{2k-1}.\]
	Expanding $y^{-\frac{s}{n}}$ in a Taylor series and applying the special case of the Cauchy product for two convergent series \cite[p.115]{ali}:
	\[\left(\sum_{k=0}^\infty a_{2k}x^{2k}\right)\left(\sum_{k=0}^\infty b_{k}x^{k}\right)=\sum_{k=0}^\infty\left(\sum_{j=0}^{\lfloor\frac{k}{2}\rfloor}a_{2j}b_{k-2j}\right)x^{k},\]
	we find
	\begin{equation}
		\int_0^\infty\frac{x^{s-1}}{1-byx^n}\mathrm{d} x=-2b\sum_{k=0}^\infty \left(\sum_{j=0}^{\lfloor\frac{k}{2}\rfloor}\frac{{(-n)}^{-k}\ln^{k-2j}(y)}{(k-2j)!}\,\Phi(b,2j,1)\right)s^{k-1}.\label{sum2}
	\end{equation}
	From (\ref{sum1}) and (\ref{sum2}), we have
	\[\lim_{s\to0}\frac{\partial^{q-1}}{\partial s^{q-1}}\left(	\int_0^\infty\frac{x^{s-1}}{1-ax}\mathrm{d} x-\int_0^\infty\frac{x^{s-1}}{1-byx^n}\mathrm{d} x\right)\]
	\[=\lim_{s\to0}\frac{\partial^{q-1}}{\partial s^{q-1}}\sum_{k=0}^\infty\underbrace{\left(-a(1+{(-1)}^{k})\,\Phi(a,k,1)+\frac{2b{(-1)}^{k}}{n^{k}}\sum_{j=0}^{\lfloor\frac{k}{2}\rfloor}\frac{\ln^{k-2j}(y)}{(k-2j)!}\,\Phi(b,2j,1)\right)}_{a_{k}}s^{k-1}.\]
	Differentiating term by term, and then computing the $(q-1)$th derivative of $s^{k-1}$ with respect to $s$, that is
	\[\frac{\partial^{q-1}}{\partial s^{q-1}} \left( a_k s^{k-1} \right) = 
	\begin{cases} 
		a_k \frac{(k-1)!}{(k-q)!} s^{k-q}, & k \geqslant q, \\ 
		0, & k < q,
	\end{cases}\]
	we conclude that
	\[\lim_{s\to0}\frac{\partial^{q-1}}{\partial s^{q-1}}\left(	\int_0^\infty\frac{x^{s-1}}{1-ax}\mathrm{d} x-\int_0^\infty\frac{x^{s-1}}{1-byx^n}\mathrm{d} x\right)=\lim_{s\to0}\sum_{k \geqslant q} a_k \frac{(k-1)!}{(k-q)!} s^{k-q}\]
	isolate the $k=q$ term to avoid the indeterminate value $0^0$
	\[=\lim_{s\to0}\left((q-1)!a_q+\sum_{k \geqslant q+1} a_k \frac{(k-1)!}{(k-q)!} s^{k-q}\right)=(q-1)!a_{q}\]
	\[=(q-1)!\left(-a(1+{(-1)}^{q})\,\Phi(a,q,1)
	+\frac{2b{(-1)}^{q}}{n^{q}}\sum_{j=0}^{\lfloor\frac{q}{2}\rfloor}\frac{\ln^{q-2j}(y)}{(q-2j)!}\,\Phi(b,2j,1)\right).\]
\end{proof}
%%%%%%%%%%%%%%%%%%%%%%%%%%%%%%%%%%%%
\begin{lemma}\label{generating function}
	Let $\mathcal{H}_k^{(p)}(b)=\sum_{m=1}^{k}\frac{b^{m-1}}{m^{p}}$, $p\in\mathbb{Z}^{+}$, $a,b=\pm1$, and $|x|<1$. Then
	\[\sum_{k=1}^{\infty} \mathcal{H}_k^{(p)}(b) x^{k}=b\frac{\operatorname{Li}_p(bx)}{1-x}.\]
\end{lemma}
\begin{proof}
	Using the summation identity \cite[p. 8]{ali}:
	\[\sum_{k=1}^\infty\sum_{m=1}^k a_k b_m=\sum_{m=1}^\infty\sum_{k=m}^\infty a_k b_m,\]
	we have
	\[\sum_{k=1}^{\infty} \mathcal{H}_k^{(p)}(b) x^{k}=\sum_{k=1}^{\infty} \sum_{m=1}^{k} \frac{b^{m-1}}{m^p} x^{k}=\sum_{m=1}^{\infty}\frac{b^{m-1}}{m^p} \left(\sum_{k=m}^{\infty} x^{k}\right)\]
	\[=\sum_{m=1}^{\infty}\frac{b^{m-1}}{m^p} \left( \frac{x^{m}}{1-x}\right)=\frac{b}{1-x}\sum_{m=1}^\infty\frac{{(bx)}^{m}}{m^p}=b\frac{\operatorname{Li}_p(bx)}{1-x}.\]
\end{proof}

%%%%%%%%%%%%%%%%%%%%%%%%%%%%%%%%%%%%%%%%%%%%%%%%%%%%%%

\section{Main results}
\begin{theorem}\label{infinity integral1}
	Let $p,q,n\in\mathbb{Z}^{+}$, and $a,b=\pm1$, with $p\neq1$ when $a^nb=1$, and $q\neq1$ when $a=1$. Then the following identity holds:
	\[\int_0^\infty\frac{\ln^{q-1}(x)\operatorname{Li}_{p}(bx^{n})}{x(1-ax)}\mathrm{d} x=-a^{n-1}b(1+{(-1)}^{q})(q-1)!\,\Phi(a^nb,p,1)\,\Phi(a,q,1)\]	
	\[+2a^n n^{-q} (q-1)!\sum_{j=0}^{\lfloor\frac{q}{2}\rfloor}\binom{p+q-2j-1}{p-1}\,\Phi(b,2j,1)\,\Phi(a^nb,p+q-2j,1)\]	
	\[+a^{n-1}b n^{-q}(q-1)!\sum_{j=2}^{n}\sum_{k=1}^{q}\binom{p+q-k-1}{p-1}a^{j}{(-1)}^{k}\,\Theta(b,k,j-1,n)\times\]
	\[\Phi\left(a^nb,p+q-k,\frac{n-j+1}{n}\right),\]	where $\Phi$ is the Lerch transcendent function, and $\Theta$ is defined in (\ref{theta}). 
\end{theorem}
\begin{remark}
	The case $p=a^nb=1$ is valid only when $n=1$, which leads to $a=b$, as the first term containing the factor $\Phi(a^nb,p,1)$ cancels with the last term in the first summation. The integral must be understood as the Cauchy principal value when $b=1$.
\end{remark}
\begin{proof}
	We have in (\ref{polylog}) that $x\leqslant1$. However, the polylogarithm can be extended to $x>1$ if the Cauchy principal value of the integral is taken into consideration, that is
	\begin{equation}
		\operatorname{Li}_p(x) = \frac{{(-1)}^{p-1}}{(p-1)!} \,\operatorname{P.V.} \int_0^1 \frac{x \ln^{p-1}(y)}{1 - xy} \, \mathrm{d} y, \quad x > 1.\label{pv}
	\end{equation}
	Using this definition, we write
	\[\int_0^\infty\frac{\ln^{q-1}(x)\operatorname{Li}_p(bx^n)}{x(1-ax)}\mathrm{d} x=\frac{{(-1)}^{p-1}}{(p-1)!}\int_0^\infty\int_0^1\frac{bx^n\ln^{q-1}(x)\ln^{p-1}(y)}{x(1-ax)(1-byx^n)}\mathrm{d} y\mathrm{d} x,\]
	where the double integral must be understood as the Cauchy principal value for the case $m=1$, as stated in (\ref{pv}). 
	Next, we swap the order of integration, which is justified by Fubini's theorem. Therefore, the main integral becomes
	\begin{equation}
		\int_0^\infty\frac{\ln^{q-1}(x)\operatorname{Li}_p(bx^n)}{x(1-ax)}\mathrm{d} x=\frac{{(-1)}^{p-1}}{(p-1)!}\int_0^1\ln^{p-1}(y)\int_0^\infty\frac{bx^n\ln^{q-1}(x)}{x(1-ax)(1-byx^n)}\mathrm{d} x\mathrm{d} y.\label{main}
	\end{equation}
	By partial fraction decomposition we have for $a,b=\pm1$:
	\begin{equation}		\frac{x^{n}}{(1-ax)(1-byx^{n})}=\frac{a^n}{1-a^n by}\left(\frac{1}{1-ax}-\frac{\sum_{j=1}^n (ax)^{j-1}}{1-byx^{n}}\right),\label{pattern}
	\end{equation}
	from which it follows that
	\[\int_0^\infty\frac{x^{n}\ln^{q-1}(x)}{x(1-ax)(1-byx^{n})}\mathrm{d} x=\frac{a^n}{1-a^n by}\int_0^\infty\frac{\ln^{q-1}(x)}{x}\left(\frac{1}{1-ax}-\frac{\sum_{j=1}^n (ax)^{j-1}}{1-byx^{n}}\right)\mathrm{d} x.\]
	Breaking up the integrand does solve the problem, as both resulting integrals would diverge. Instead, we separate the first term of the inner sum and then change the order of integration and summation. This yields
	\[\int_0^\infty\frac{x^{n}\ln^{q-1}(x)}{x(1-ax)(1-byx^{n})}\mathrm{d} x=\frac{a^n}{1-a^n by}\int_0^\infty\frac{\ln^{q-1}(x)}{x}\left(\frac{1}{1-ax}-\frac{1}{1-byx^n}\right)\mathrm{d} x\]
	\[-\frac{a^n}{1-a^n by}\sum_{j=2}^n a^{j-1}\int_0^\infty\frac{x^{j-2}\ln^{q-1}(x)}{1-byx^n}\mathrm{d} x.\]
	With these adjustments, both integrals are now convergent. The first integral appears in Lemma \ref{tough integral}. For the second integral, we let $yx^n=t^n$, apply the binomial expansion for $\ln^{q-1}(y^{-\frac{1}{n}}t)$, and then make use of Lemma \ref{theta integral}, we find
	\begin{equation}
		\int_0^\infty\frac{x^{j-2}\ln^{q-1}(x)}{1-byx^{n}}\mathrm{d} x=\frac{{(-1)}^{q-1}(q-1)!}{n^{q}}\sum_{k=1}^{q}\frac{\Theta(b,k,j-1,n)}{(q-k)!}y^{\frac{-j+1}{n}}\ln^{q-k}(y).\label{first integral}
	\end{equation}
	Collecting these two integrals, we conclude
	\[\int_0^\infty\frac{ x^{n-1}\ln^{q-1}(x)}{(1-ax)(1-byx^{n})}\mathrm{d} x
	=-\frac{a^{n-1}(q-1)!}{1-a^n by}(1+{(-1)}^{q})\,\Phi(a,q,1)\]
	\[+\frac{2a^nb{(-1)}^{q} (q-1)!}{(1-a^n by)n^{q}}\sum_{j=0}^{\lfloor\frac{q}{2}\rfloor}\frac{\Phi(b,2j,1)}{(q-2j)!}\ln^{q-2j}(y)\]
	\[+\frac{a^{n-1}{(-1)}^{q}(q-1)!}{(1-a^n by)n^{q}}\sum_{j=2}^n \sum_{k=1}^{q}a^{j}\frac{\Theta(b,k,j-1,n)}{(q-k)!}y^{\frac{-j+1}{n}}\ln^{q-k}(y).\]
	Returning to (\ref{main}) and then using Lemma \ref{lerch integral} completes the proof.
\end{proof}
A related integral can be found by substituting $x\to 1/x$, that is
\[\int_0^\infty\frac{\ln^{q-1}(x)\operatorname{Li}_{p}(bx^{-n})}{1-ax}\mathrm{d} x={(-1)}^{q}a\int_0^\infty\frac{\ln^{q-1}(x)\operatorname{Li}_{p}(bx^{n})}{x(1-ax)}\mathrm{d} x.\]
%%%%%%%%%%%%%%%%%%%%%%%%%%%%%%%%%%%%%%%%%%%

\begin{theorem}\label{infinity integral2}
	Let $p,q,n\in\mathbb{Z}^{+}$, and $a,b=\pm1$, with $p\neq1$ when $a^nb=1$, and $q\neq1$ when $a=1$. Then the following identity holds:
	\[\int_0^\infty\frac{\ln^{q-1}(x)\operatorname{Li}_p(b{x}^{2n})}{1-ax^2}\mathrm{d} x=-a^{n-1}b 2^{-q}(q-1)!(1+a{(-1)}^{q})\,\Phi(a,q,\tfrac12)\,\Phi(a^nb,p,1)\]
	\[+a^{n-1}b {(2n)}^{-q}(q-1)!\sum_{j=1}^{n}\sum_{k=1}^{q}\binom{p+q-k-1}{p-1}a^j{(-1)}^{k}\,\Theta(b,k,2j-1,2n)\times\]
	\[\Phi\left(a^nb,p+q-k,\frac{2n-2j+1}{2n}\right),\]
	where $\Phi$ is the Lerch transcendent function, and $\Theta$ is defined in (\ref{theta}). The integral must be understood as the Cauchy principal value when $b=1$.
\end{theorem}
\begin{remark}
	The case $p=a^nb=1$ is valid only when $n=1$, which leads to $a=b$ , as the addition of the first term containing the factor $\Phi(a^nb,p,1)$ and the $k=q$ term in the summation simplifies to \[2^{p-q}b(1+a{(-1)}^q)(q-1)!\Phi(a,q,\tfrac12)\eta(p).\] 
\end{remark}
\begin{proof}
	The proof follows from applying the same approach as in the previous proof, and using the following two identities:		
	\[\frac{x^{2n}}{(1-ax^2)(1-byx^{2n})}=\frac{a^n}{1-a^n by}\left(\frac{1}{1-ax^2}-\frac{\sum_{j=1}^n (ax^2)^{j-1}}{1-byx^{2n}}\right);\]
	\[\int_0^\infty\frac{x^{2j-2}\ln^{q-1}(x)}{1-byx^{2n}}\mathrm{d} x=\frac{{(-1)}^{q-1}(q-1)!}{{(2n)}^{q}}\sum_{k=1}^{q}\frac{\Theta(b,k,2j-1,2n)}{(q-k)!}y^{\frac{-2j+1}{2n}}\ln^{q-k}(y).\]
	which respectively follow from replacing $x$ with $x^2$ in (\ref{pattern}), and replacing $(j,n)$ with $(2j,2n)$ in (\ref{first integral}). The only difference in this proof is that it's unnecessary to isolate the first term of the inner sum $\sum_{j=1}^n (ax^2)^{j-1}$, which makes the proof even easier and shorter, as divergence issues do not arise. We then apply Lemma \ref{theta integral}.
\end{proof}

To extend the range of $n$ in Theorem \ref{infinity integral2} to negative integers, we make a change of variable $x\to1/x$, which yields
\[\int_0^\infty\frac{\ln^{q-1}(x)\operatorname{Li}_p(b{x}^{-2n})}{1-ax^2}\mathrm{d} x={(-1)}^{q}a\int_0^\infty\frac{\ln^{q-1}(x)\operatorname{Li}_p(b{x}^{2n})}{1-ax^2}\mathrm{d} x.\]

%%%%%%%%%%%%%%%%%%%%%%%%%%%%%%%%%%%%%%%%%%%%%%%%%%%%%%%%%%%%%
\begin{theorem}\label{from zero to one integral1}
	Let $p,q,n\in\mathbb{Z}^{+}$, $p+q$ odd, and $a,b=\pm1$, with $p\neq1$ when $a^nb=1$, and $q\neq1$ when $a=1$. Then the following two identities hold:
	\begin{enumerate}
		\item[i)]
		\[\int_0^1\frac{\ln^{q-1}(x)\operatorname{Li}_{p}(bx^{n})}{1-ax}\mathrm{d} x=\frac12 a b{(-1)}^{q}(q-1)!n^{-q}\,\Phi(b,p+q,1)\]
		\[-\frac12a^{n}b(1+{(-1)}^{q})(q-1)!\,\Phi(a,q,1)\,\Phi(a^nb,p,1)\]
		\[+a^{n-1}n^{-q} (q-1)!\sum_{j=0}^{\lfloor\frac{q}{2}\rfloor}\binom{p+q-2j-1}{p-1}\,\Phi(b,2j,1)\,\Phi(a^nb,p+q-2j,1)\]
		\[+bn^{p}(q-1)!\sum_{k=0}^{\lfloor\frac{p}{2}\rfloor}\binom{p+q-2k-1}{q-1}n^{-2k}\,\Phi(b,2k,1)\,\Phi(a,p+q-2k,1)\]
		\[+\frac12a^{n}b n^{-q}(q-1)!\sum_{j=2}^{n}\sum_{k=1}^{q}\binom{p+q-k-1}{p-1}a^{j}{(-1)}^{k}\,\Theta(b,k,j-1,n)\times\]
		\[\Phi\left(a^nb,p+q-k,\frac{n-j+1}{n}\right);\]
		\item[ii)]
		\[\int_0^1\frac{\ln^{q-1}(x)\operatorname{Li}_{p}(bx^{-n})}{1-ax}\mathrm{d} x=\frac12 ab(q-1)!n^{-q}\,\Phi(b,p+q,1)\]
		\[-\frac12a^{n}b(1+{(-1)}^{q})(q-1)!\,\Phi(a^nb,p,1)\,\Phi(a,q,1)\]
		\[+a^{n-1}{(-n)}^{-q}(q-1)!\sum_{j=0}^{\lfloor\frac{q}{2}\rfloor}\binom{p+q-2j-1}{p-1}\,\Phi(b,2j,1)\,\Phi(a^nb,p+q-2j,1)\]
		\[+b{(-n)}^{p}(q-1)!\sum_{k=0}^{\lfloor\frac{p}{2}\rfloor}\binom{p+q-2k-1}{q-1}n^{-2k}\,\Phi(b,2k,1)\,\Phi(a,p+q-2k,1)\]
		\[+\frac12 a^{n}b {(-n)}^{-q}(q-1)!\sum_{j=2}^{n}\sum_{k=1}^{q}\binom{p+q-k-1}{p-1}a^{j}{(-1)}^{k}\,\Theta(b,k,j-1,n)\times\]
		\[\Phi\left(a^nb,p+q-k,\frac{n-j+1}{n}\right),\]
	\end{enumerate}		
	where $\Phi$ is the Lerch transcendent function, and $\Theta$ is defined in (\ref{theta}). The integral in part (ii) must be understood as the Cauchy principal value when $b=1$. 
\end{theorem}
\begin{remark}
	The case $p=a^nb=1$ is valid only when $n=1$, as the second term containing the factor $\Phi(a^nb,p,1)$ cancels with the last term in the first summation.
\end{remark}
\begin{proof}
	We begin with splitting the integral at $x=1$,
	\[\int_0^\infty\frac{\ln^{q-1}(x)\operatorname{Li}_{p}(bx^{n})}{x(1-ax)}\mathrm{d} x=\int_0^1\frac{\ln^{q-1}(x)\operatorname{Li}_{p}(bx^{n})}{x(1-ax)}\mathrm{d} x+\int_1^\infty\frac{\ln^{q-1}(x)\operatorname{Li}_{p}(bx^{n})}{x(1-ax)}\mathrm{d} x.\]
	For the first integral, we use $\frac{1}{x(1-ax)}=\frac{1}{x}+\frac{a}{1-ax}$. For the second integral, we substitute $x\to 1/x$. Therefore, the integral becomes
	\[\int_0^\infty\frac{\ln^{q-1}(x)\operatorname{Li}_{p}(bx^{n})}{1-ax}\mathrm{d} x\]
	\[=\int_0^1\frac{\ln^{q-1}(x)\operatorname{Li}_{p}(bx^{n})}{x}\mathrm{d} x+a\int_0^1\frac{\ln^{q-1}(x)}{1-ax}\left(\operatorname{Li}_{p}(bx^{n})+{(-1)}^{q} \operatorname{Li}_{p}\left(\frac{b}{x^n}\right)\right)\mathrm{d} x.\]
	The integral on the left-hand side appears in Theorem \ref{infinity integral1}. For the first integral on the right-hand side, expanding $\operatorname{Li}_{p}(bx^{n})$ in its Taylor series, then reverting the order of integration and summation, justified by the uniform convergence theorem, we find
	\[\int_0^1\frac{\ln^{q-1}(x)\operatorname{Li}_p(bx^n)}{x}\mathrm{d} x=\sum_{k=1}^\infty\frac{b^k}{k^p}\int_0^1 x^{nk-1}\ln^{q-1}(x)\mathrm{d} x\]
	\begin{equation}
		={(-1)}^{q-1}(q-1)!\sum_{k=1}^\infty\frac{b^k}{k^p(nk)^{q}}=b{(-1)}^{q-1}(q-1)!\,n^{-q}\,\Phi(b,p+q,1).\label{ff}
	\end{equation}
	Substituting this result into our expression and noting that ${(-1)}^{q}={(-1)}^{p-1}$ due to the assumption that $p+q$ is odd, we obtain
	\[\int_0^1\frac{\ln^{q-1}(x)}{1-ax}\left(\operatorname{Li}_{p}(bx^{n})-{(-1)}^{p} \operatorname{Li}_{p}\left(\frac{b}{x^n}\right)\right)\mathrm{d} x=ab{(-1)}^{q}(q-1)!\,n^{-q}\,\Phi(b,p+q,1)\]
	\[-a^{n}b(1+{(-1)}^{q})(q-1)!\,\Phi(a^nb,p,1)\,\Phi(a,q,1)\]
	\[+2a^{n-1} n^{-q} (q-1)!\sum_{j=0}^{\lfloor\frac{q}{2}\rfloor}\binom{p+q-2j-1}{p-1}\,\Phi(b,2j,1)\,\Phi(a^nb,p+q-2j,1)\]		
	\[+a^{n}b n^{-q}(q-1)!\sum_{j=2}^{n}\sum_{k=1}^{q}\binom{p+q-k-1}{p-1}a^{j}{(-1)}^{k}\,\Theta(b,k,j-1,n)\times\]
	\begin{equation}
		\Phi\left(a^nb,p+q-k,\frac{n-j+1}{n}\right).\label{yum1}
	\end{equation}
	To establish another relationship, we substitute $(c,x)$ with $(b,x^{n})$ in Lemma \ref{Li(bx)+Li(b/x)}, multiply through by $\frac{\ln^{q-1}(x)}{1-ax}$, and integrate for $x\in(0,1)$. Reordering the summation and integration, which is justified by the uniform convergence theorem, and then applying the integral form given in (\ref{lerch integral}), we find
	\[\int_0^1\frac{\ln^{q-1}(x)}{1-ax}\left(\operatorname{Li}_{p}(bx^{n})+{(-1)}^{p} \operatorname{Li}_{p}\left(\frac{b}{x^n}\right)\right)\mathrm{d} x\]
	\begin{equation}
		=2b n^{p}(q-1)!\sum_{k=0}^{\lfloor\frac{p}{2}\rfloor}\binom{p+q-2k-1}{q-1}n^{-2k}\,\Phi(b,2k,1)\,\Phi(a,p+q-2k,1).\label{yum2}
	\end{equation}
	The proof follows from (\ref{yum1}) and (\ref{yum2}).
\end{proof}

%%%%%%%%%%%%%%%%%%%%%%%%%%%%%%%%%%%%%%%%%
\begin{corollary} Let $a=1$ and $b=\pm1$ in Theorem \ref{from zero to one integral1}. Then for $p+q$ odd
	\begin{enumerate}
		\item[i)]
		\[\int_0^1\frac{\ln^{q-1}(x)\operatorname{Li}_{p}(x^{n})}{1-x}\mathrm{d} x\]
		\[=\frac12 {(-1)}^{q}(q-1)!n^{-q}\zeta(p+q)-\frac12(1+{(-1)}^{q})(q-1)!\zeta(p)\zeta(q)\]
		\[+n^{p}(q-1)!\sum_{k=0}^{\lfloor\frac{p}{2}\rfloor}\binom{p+q-2k-1}{q-1}n^{-2k}\zeta(2k)\zeta(p+q-2k)\]
		\[+n^{-q}(q-1)!\sum_{j=0}^{\lfloor\frac{q}{2}\rfloor}\binom{p+q-2j-1}{p-1}\zeta(2j)\zeta(p+q-2j)\]
		\[+\frac12 n^{-q}(q-1)!\sum_{j=2}^{n}\sum_{k=1}^{q}\binom{p+q-k-1}{p-1}{(-1)}^{k}\,\Theta(1,k,j-1,n)\times\]
		\[\Phi\left(1,p+q-k,\frac{n-j+1}{n}\right);\]
		\item[ii)]
		\[\int_0^1\frac{\ln^{q-1}(x)\operatorname{Li}_{p}(-x^{n})}{1-x}\mathrm{d} x\]
		\[=-\frac12 {(-1)}^{q}(q-1)!n^{-q}\eta(p+q)+\frac12(1+{(-1)}^{q})(q-1)!\eta(p)\zeta(q)\]
		\[-n^{p}(q-1)!\sum_{k=0}^{\lfloor\frac{p}{2}\rfloor}\binom{p+q-2k-1}{q-1}n^{-2k}\eta(2k)\zeta(p+q-2k)\]
		\[+n^{-q}(q-1)!\sum_{j=0}^{\lfloor\frac{q}{2}\rfloor}\binom{p+q-2j-1}{p-1}\eta(2j)\eta(p+q-2j)\]
		\[-\frac12 n^{-q}(q-1)!\sum_{j=2}^{n}\sum_{k=1}^{q}\binom{p+q-k-1}{p-1}{(-1)}^{k}\Theta(-1,k,j-1,n)\times\]
		\[\Phi\left(-1,p+q-k,\frac{n-j+1}{n}\right).\]
	\end{enumerate}
\end{corollary}

%%%%%%%%%%%%%%%%%%%%%%%%%%%%%%%%%%%%%%%%%%%%%%%%%%
\begin{corollary}
	Let $a=-1$ and $b=\pm1$ in Theorem \ref{from zero to one integral1}. Then for $p+q$ odd
	\begin{enumerate}
		\item[i)]
		\[\int_0^1\frac{\ln^{q-1}(x)\operatorname{Li}_{p}(x^{n})}{1+x}\mathrm{d} x=-\frac12{(-1)}^{q}(q-1)!n^{-q}\zeta(p+q)\]
		\[-\frac12{(-1)}^{n}(1+{(-1)}^{q})(q-1)!\eta(q)\,\Phi({(-1)}^{n},p,1)\]
		\[+n^{p}(q-1)!\sum_{k=0}^{\lfloor\frac{p}{2}\rfloor}\binom{p+q-2k-1}{q-1}n^{-2k}\zeta(2k)\eta(p+q-2k)\]
		\[-{(-1)}^{n}n^{-q} (q-1)!\sum_{j=0}^{\lfloor\frac{q}{2}\rfloor}\binom{p+q-2j-1}{p-1}\zeta(2j)\,\Phi({(-1)}^{n},p+q-2j,1)\]
		\[+\frac12{(-1)}^{n}n^{-q}(q-1)!\sum_{j=2}^{n}\sum_{k=1}^{q}\binom{p+q-k-1}{p-1}{(-1)}^{j}{(-1)}^{k}\,\Theta(1,k,j-1,n)\times\]
		\[\Phi\left({(-1)}^{n},p+q-k,\frac{n-j+1}{n}\right);\]
		\item[ii)]
		\[\int_0^1\frac{\ln^{q-1}(x)\operatorname{Li}_{p}(-x^{n})}{1+x}\mathrm{d} x=\frac12 {(-1)}^{q}n^{-q}(q-1)!\eta(p+q)\]
		\[+\frac12{(-1)}^{n}(1+{(-1)}^{q})(q-1)!\eta(q)\,\Phi({(-1)}^{n-1},p,1)\]
		\[-n^{p}(q-1)!\sum_{k=0}^{\lfloor\frac{p}{2}\rfloor}\binom{p+q-2k-1}{q-1}n^{-2k}\eta(2k)\eta(p+q-2k)\]
		\[-{(-1)}^{n}n^{-q} (q-1)!\sum_{j=0}^{\lfloor\frac{q}{2}\rfloor}\binom{p+q-2j-1}{p-1}\eta(2j)\,\Phi({(-1)}^{n-1},p+q-2j,1)\]
		\[-\frac12{(-1)}^{n}n^{-q}(q-1)!\sum_{j=2}^{n}\sum_{k=1}^{q}\binom{p+q-k-1}{p-1}{(-1)}^{j}{(-1)}^{k}\,\Theta(-1,k,j-1,n)\times\]
		\[\Phi\left({(-1)}^{n-1},p+q-k,\frac{n-j+1}{n}\right).\]
	\end{enumerate}
\end{corollary}

%%%%%%%%%%%%%%%%%%%%%%%%%%%%%%%%%%%%%%%%%%%
\begin{theorem}\label{from zero to one integral2}
	Let $p,q,n\in\mathbb{Z}^{+}$, $b=\pm1$, and $a={(-1)}^{p+q-1}$, with $p\neq1$ when $a^nb=1$, and $q\neq1$ when $a=1$. Then the following two identities hold:
	\begin{enumerate}
		\item[i)]
		\[\int_0^1\frac{\ln^{q-1}(x)\operatorname{Li}_p(b{x}^{2n})}{1-ax^2}\mathrm{d} x\]
		\[=-\frac12a^{n-1}b 2^{-q}(q-1)!(1-{(-1)}^{p})\,\Phi(a,q,\tfrac12)\,\Phi(a^nb,p,1)\]
		\[+ab 2^{-q} n^{p}(q-1)!\sum_{k=0}^{\lfloor\frac{p}{2}\rfloor}\binom{p+q-2k-1}{q-1}n^{-2k}\,\Phi(b,2k,1)\,\Phi(a,p+q-2k,\tfrac{1}{2})\]
		\[+\frac12 a^{n-1}b {(2n)}^{-q}(q-1)!\sum_{j=1}^{n}\sum_{k=1}^{q}\binom{p+q-k-1}{p-1}a^j{(-1)}^{k}\,\Theta(b,k,2j-1,2n)\times\]
		\[\Phi\left(a^nb,p+q-k,\frac{2n-2j+1}{2n}\right);\]
		\item[ii)]
		\[\int_0^1\frac{\ln^{q-1}(x)\operatorname{Li}_p(bx^{-2n})}{1-ax^2}\mathrm{d} x\]
		\[=-\frac12 a^n b {(-2)}^{-q}(q-1)!(1-{(-1)}^{p})\,\Phi(a,q,\tfrac12)\,\Phi(a^nb,p,1)\]
		\[-b{(-2)}^{-q}n^{p}(q-1)!\sum_{k=0}^{\lfloor\frac{p}{2}\rfloor}\binom{p+q-2k-1}{q-1}n^{-2k}\,\Phi(b,2k,1)\,\Phi(a,p+q-2k,\tfrac{1}{2})\]
		\[+\frac12a^{n}b{(-2n)}^{-q}(q-1)!\sum_{j=1}^{n}\sum_{k=1}^{q}\binom{p+q-k-1}{p-1}a^j{(-1)}^{k}\,\Theta(b,k,2j-1,2n)\times\]
		\[\Phi\left(a^nb,p+q-k,\frac{2n-2j+1}{2n}\right),\]
	\end{enumerate}						
	where $\Phi$ is the Lerch transcendent function, and $\Theta$ is defined in (\ref{theta}). The integral in part (ii) must be understood as the Cauchy principal value when $b=1$. 
\end{theorem}
\begin{remark}
	The case $p=a^nb=1$ is valid only when $n=1$, which leads to $a=b$ , as the addition of the first term containing the factor $\Phi(a^nb,p,1)$ and the $k=q$ term in the second summation simplifies to \[2^{p-q-1}b(1-{(-1)}^p)(q-1)!\Phi(a,q,\tfrac12)\eta(p).\]
\end{remark}
\begin{proof} 
	We proceed by following the same steps as in the proof of Theorem \ref{from zero to one integral1}, splitting the integral at $x=1$ and making the substitution $x\to 1/x$ in the second integral,
	\[\int_0^\infty\frac{\ln^{q-1}(x)\operatorname{Li}_p(bx^{2n})}{1-ax^2}\mathrm{d} x=\int_0^1\frac{\ln^{q-1}(x)}{1-ax^2}\left(\operatorname{Li}_p(bx^{2n})+a{(-1)}^{q}\operatorname{Li}_p\left(\frac{b}{x^{2n}}\right)\right)\mathrm{d} x.\]
	Then we use the result from Theorem \ref{infinity integral2} and substitute $a{(-1)}^{q}={(-1)}^{p-1}$,
	\[\int_0^1\frac{\ln^{q-1}(x)}{1-ax^2}\left(\operatorname{Li}_p(bx^{2n})-{(-1)}^{p}\operatorname{Li}_p\left(\frac{b}{x^{2n}}\right)\right)\mathrm{d} x\]
	\[=-\frac{a^{n-1}b(q-1)!}{2^{q}}(1-{(-1)}^{p})\,\Phi(a,q,\tfrac12)\,\Phi(a^nb,p,1)\]
	\[+\frac{a^{n-1}b (q-1)!}{{(2n)}^{q}}\sum_{j=1}^{n}\sum_{k=1}^{q}\binom{p+q-k-1}{p-1}a^j{(-1)}^{k}\,\Theta(b,k,2j-1,2n)\times\]
	\begin{equation}
		\Phi\left(a^nb,p+q-k,\frac{2n-2j+1}{2n}\right).\label{eq1}
	\end{equation}	
	Another relationship can be established by employing Lemma \ref{Li(bx)+Li(b/x)},
	\[\int_0^1\frac{\ln^{q-1}(x)}{1-ax^2}\left(\operatorname{Li}_p(bx^{2n})+{(-1)}^{p}\operatorname{Li}_p\left(\frac{b}{x^{2n}}\right)\right)\mathrm{d} x\]
	\[=2b\sum_{k=0}^{\lfloor{\frac{p}{2}}\rfloor} \frac{{(2n)}^{p-2k}}{(p-2k)!}\,\Phi(c,2k,1)\int_0^1\frac{\ln^{p+q-2k-1}}{1-ax^2}\mathrm{d} x\]
	\begin{equation}
		=\frac{ab n^{p}(q-1)!}{2^{q-1}}\sum_{k=0}^{\lfloor\frac{p}{2}\rfloor}\binom{p+q-2k-1}{q-1}n^{-2k}\Phi(b,2k,1)\Phi(a,p+q-2k,\tfrac{1}{2}).\label{eq2}
	\end{equation}
	The proof follows from (\ref{eq1}) and (\ref{eq2}).
\end{proof}

%%%%%%%%%%%%%%%%%%%%%%%%%%%%%%%%%%%%%%%%%%%%%%%%%%%%%%%%%%%%
\begin{corollary}
	Let $a=1$, and $b=\pm1$ in Theorem \ref{from zero to one integral2}. Then for $p+q$ odd
	\begin{enumerate}
		\item[i)]
		\[\int_0^1\frac{\ln^{q-1}(x)\operatorname{Li}_p(x^{2n})}{1-x^2}\mathrm{d} x=-\frac12(q-1)!(1-{(-1)}^{p})\lambda(q)\zeta(p)\]
		\[+{(2n)}^{p}(q-1)!\sum_{k=0}^{\lfloor\frac{p}{2}\rfloor}\binom{p+q-2k-1}{q-1}{(2n)}^{-2k}\zeta(2k)\lambda(p+q-2k)\]
		\[+\frac12{(2n)}^{-q}(q-1)!\sum_{j=1}^{n}\sum_{k=1}^{q}\binom{p+q-k-1}{p-1}{(-1)}^{k}\,\Theta(1,k,2j-1,2n)\times\]
		\[\Phi\left(1,p+q-k,\frac{2n-2j+1}{2n}\right);\]
		\item[ii)]
		\[\int_0^1\frac{\ln^{q-1}(x)\operatorname{Li}_p(-{x}^{2n})}{1-x^2}\mathrm{d} x=\frac12(q-1)!(1-{(-1)}^{p})\lambda(q)\eta(p)\]
		\[-{(2n)}^{p}(q-1)!\sum_{k=0}^{\lfloor\frac{p}{2}\rfloor}\binom{p+q-2k-1}{q-1}{(2n)}^{-2k}\eta(2k)\lambda(p+q-2k)\]
		\[-\frac12 {(2n)}^{-q}(q-1)!\sum_{j=1}^{n}\sum_{k=1}^{q}\binom{p+q-k-1}{p-1}{(-1)}^{k}\,\Theta(-1,k,2j-1,2n)\times\]
		\[\Phi\left(-1,p+q-k,\frac{2n-2j+1}{2n}\right).\]
	\end{enumerate}
\end{corollary}
%%%%%%%%%%%%%%%%%%%%%%%%%%%%%%%%%%%%%%%%%%%%%%%%%
\begin{corollary}
	Let $a=-1$, and $b=\pm1$ in Theorem \ref{from zero to one integral2}. Then for $p+q$ even
	\begin{enumerate}
		\item[i)]
		\[\int_0^1\frac{\ln^{q-1}(x)\operatorname{Li}_p({x}^{2n})}{1+x^2}\mathrm{d} x\]
		\[=-\frac12{(-1)}^{n-1}(q-1)!(1-{(-1)}^{p})\beta(q)\,\Phi({(-1)}^{n},p,1)\]
		\[-{(2n)}^{p}(q-1)!\sum_{k=0}^{\lfloor\frac{p}{2}\rfloor}\binom{p+q-2k-1}{q-1}{(2n)}^{-2k}\zeta(2k)\beta(p+q-2k)\]
		\[-\frac12{(-1)}^{n}{(2n)}^{-q}(q-1)!\sum_{j=1}^{n}\sum_{k=1}^{q}\binom{p+q-k-1}{p-1}{(-1)}^{k+j}\,\Theta(1,k,2j-1,2n)\times\]
		\[\Phi\left({(-1)}^{n},p+q-k,\frac{2n-2j+1}{2n}\right);\] 
		\item[ii)]
		\[\int_0^1\frac{\ln^{q-1}(x)\operatorname{Li}_p(-{x}^{2n})}{1+x^2}\mathrm{d} x\]
		\[=-\frac12{(-1)}^{n}(q-1)!(1-{(-1)}^{p})\beta(q)\,\Phi({(-1)}^{n-1},p,1)\]
		\[+{(2n)}^{p}(q-1)!\sum_{k=0}^{\lfloor\frac{p}{2}\rfloor}\binom{p+q-2k-1}{q-1}{(2n)}^{-2k}\eta(2k)\beta(p+q-2k)\]
		\[+\frac12{(-1)}^{n}{(2n)}^{-q}(q-1)!\sum_{j=1}^{n}\sum_{k=1}^{q}\binom{p+q-k-1}{p-1}{(-1)}^{k+j}\,\Theta(-1,k,2j-1,2n)\times\]
		\[\Phi\left({(-1)}^{n-1},p+q-k,\frac{2n-2j+1}{2n}\right).\]
	\end{enumerate}
\end{corollary}

\section{Euler sums results}
%%%%%%%%%%%%%%%%%%%%%%%%%%%%%%%%%%%%%%%%%%%
\begin{theorem} \label{thm of H_n}
	Let $p,q,n\in\mathbb{Z}^{+}$, $p+q$ odd, and $a,b=\pm1$, with $p\neq1$ when $b=1$, and $q\neq1$ when $a=1$. Then the following identity holds:
	\[\sum_{k=1}^\infty\frac{a^k\, \mathcal{H}_{nk}^{(p)}(b)}{k^q}=\frac12 ab^{n-1}n^{-p}\,\Phi(ab^n,p+q,1)+\frac12a(1-{(-1)}^{p})\,\Phi(a,q,1)\,\Phi(b,p,1)\]
	\[-ab^n{(-n)}^{q}\sum_{k=0}^{\lfloor\frac{q}{2}\rfloor}\binom{p+q-2k-1}{p-1}n^{-2k}\,\Phi(ab^n,2k,1)\,\Phi(b,p+q-2k,1)\]
	\[+b^{n-1}{(-n)}^{-p}\sum_{j=0}^{\lfloor\frac{p}{2}\rfloor}\binom{p+q-2j-1}{q-1}\,\Phi(ab^n,2j,1)\,\Phi(a,p+q-2j,1)\]
	\[+\frac{a}{2}{(-n)}^{-p}\sum_{j=2}^{n}\sum_{k=1}^{p}\binom{p+q-k-1}{q-1}b^{j}{(-1)}^{k}\,\Theta(ab^n,k,j-1,n)\times\]
	\[\Phi\left(a,p+q-k,\frac{n-j+1}{n}\right),\]
	where $\mathcal{H}_{nk}^{(p)}(b)=\sum_{m=1}^{nk}\frac{b^{m-1}}{m^p}$, $\Phi$ is the Lerch function, and $\Theta$ is defined in (\ref{theta}).
\end{theorem}
\begin{remark}
	The case $p=b=1$ is valid only when $n=1$ because the first term containing the factor $\Phi(b,p,1)$ cancels with the last term in the first summation.
\end{remark}
\begin{proof} Writing the integral form of $1/m^{p}$ as
	\[\frac{1}{m^p}=\frac{{(-1)}^{p-1}}{(p-1)!}\int_0^1x^{m-1}\ln^{p-1}(x)\mathrm{d} x\]
	and then reverting the order of integration and summation, we have
	\[\mathcal{H}_{nk}^{(p)}(b)=\sum_{m=1}^{nk}\frac{b^{m-1}}{m^p}=\frac{{(-1)}^{p-1}}{(p-1)!}\int_0^1 \ln^{p-1}(x)\sum_{m=1}^{nk} (bx)^{m-1}\mathrm{d} x\]
	\[=\frac{{(-1)}^{p-1}}{(p-1)!}\int_0^1 \ln^{p-1}(x)\frac{1-(bx)^{nk}}{1-bx}\mathrm{d} x.\]
	Using this integral form and then changing the order of integration and summation, which is justified by the uniform convergence theorem, we write
	\[\sum_{k=1}^\infty\frac{a^k\, \mathcal{H}_{nk}^{(p)}(b)}{k^q}=\frac{{(-1)}^{p-1}}{(p-1)}\int_0^1\frac{\ln^{p-1}(x)}{1-bx}\sum_{k=1}^\infty\frac{a^k-(ab^nx^n)^k}{n^q}\mathrm{d} x\]
	\[=\frac{{(-1)}^{p-1}}{(p-1)}\int_0^1\frac{\ln^{p-1}(x)}{1-bx}\left(a\,\Phi(a,q,1)-\operatorname{Li}_q\left(ab^nx^n\right)\right)\mathrm{d} x\]	
	\[=a\,\Phi(a,q,1)\,\Phi(b,p,1)-\frac{{(-1)}^{p-1}}{(p-1)!}\int_0^1\frac{\ln^{p-1}(x)\operatorname{Li}_q\left(ab^nx^n\right)}{1-bx}\mathrm{d} x.\]
	The proof follows from Theorem \ref{from zero to one integral1} with $(a,b)$ replaced with $(b,ab^n)$.
\end{proof}
%%%%%%%%%%%%%%%%%%%%%%%%%
\begin{corollary}
	Let $(a,b)=(\pm1,1)$ in Theorem \ref{Euler sums1}. Then for $p+q$ odd
	\begin{enumerate}
		\item[i)]
		\[\sum_{k=1}^\infty\frac{H_{nk}^{(p)}}{k^q}=\frac12(1-{(-1)}^{p})\zeta(q)\zeta(p)+\frac12 n^{-p}\zeta(p+q)\]
		\[+{(-n)}^{-p}\sum_{j=0}^{\lfloor\frac{p}{2}\rfloor}\binom{p+q-2j-1}{q-1}\zeta(2j)\zeta(p+q-2j)\]
		\[-{(-n)}^{q}\sum_{k=0}^{\lfloor\frac{q}{2}\rfloor}\binom{p+q-2k-1}{p-1}n^{-2k}\zeta(2k)\zeta(p+q-2k)\]
		\[+\frac12 {(-n)}^{-p}\sum_{j=2}^{n}\sum_{k=1}^{p}\binom{p+q-k-1}{q-1}{(-1)}^{k}\,\Theta(1,k,j-1,n)\times\]
		\[\Phi\left(1,p+q-k,\frac{n-j+1}{n}\right);\]
		\item[ii)]
		\[\sum_{k=1}^\infty\frac{{(-1)}^{k} H_{nk}^{(p)}}{k^q}=-\frac12(1-{(-1)}^{p})\eta(q)\zeta(p)-\frac12 n^{-p}\eta(p+q)\]
		\[+{(-n)}^{-p}\sum_{j=0}^{\lfloor\frac{p}{2}\rfloor}\binom{p+q-2j-1}{q-1}\eta(2j)\eta(p+q-2j)\]
		\[+{(-n)}^{q}\sum_{k=0}^{\lfloor\frac{q}{2}\rfloor}\binom{p+q-2k-1}{p-1}n^{-2k}\eta(2k)\zeta(p+q-2k)\]
		\[-\frac12 {(-n)}^{-p}\sum_{j=2}^{n}\sum_{k=1}^{p}\binom{p+q-k-1}{q-1}{(-1)}^{k}\,\Theta(-1,k,j-1,n)\times\]
		\[\Phi\left(-1,p+q-k,\frac{n-j+1}{n}\right);\]
		\end{enumerate}
	\end{corollary}
		\begin{corollary}
			Let $(a,b)=(\pm1,-1)$ in Theorem \ref{Euler sums1}. Then for $p+q$ odd
			\begin{enumerate}
		\item[i)]
		\[\sum_{k=1}^\infty\frac{\overline{H}_{nk}^{(p)}}{k^q}=\frac12(1-{(-1)}^{p})\zeta(q)\eta(p)-\frac12 {(-1)}^{n}n^{-p}\,\Phi({(-1)}^{n},p+q,1)\]
		\[-{(-1)}^{n}{(-n)}^{-p}\sum_{j=0}^{\lfloor\frac{p}{2}\rfloor}\binom{p+q-2j-1}{q-1}\,\Phi({(-1)}^{n},2j,1)\zeta(p+q-2j)\]
		\[-{(-1)}^{n} {(-n)}^{q}\sum_{k=0}^{\lfloor\frac{q}{2}\rfloor}\binom{p+q-2k-1}{p-1}n^{-2k}\,\Phi({(-1)}^{n},2k,1)\eta(p+q-2k)\]
		\[+\frac12 {(-n)}^{-p}\sum_{j=2}^{n}\sum_{k=1}^{p}\binom{p+q-k-1}{q-1}{(-1)}^{k+j}\,\Theta({(-1)}^{n},k,j-1,n)\times\]
		\[\Phi\left(1,p+q-k,\frac{n-j+1}{n}\right);\]
		\item[ii)]
		\[\sum_{k=1}^\infty\frac{{(-1)}^{k}\,\overline{H}_{nk}^{(p)}}{k^q}=-\frac12(1-{(-1)}^{p})\eta(q)\eta(p)+\frac12 {(-1)}^{n}n^{-p}\,\Phi({(-1)}^{n-1},p+q,1)\]
		\[-{(-1)}^{n}{(-n)}^{-p}\sum_{j=0}^{\lfloor\frac{p}{2}\rfloor}\binom{p+q-2j-1}{q-1}\,\Phi({(-1)}^{n-1},2j,1)\eta(p+q-2j)\]
		\[+{(-1)}^{n}{(-n)}^{q}\sum_{k=0}^{\lfloor\frac{q}{2}\rfloor}\binom{p+q-2k-1}{p-1}n^{-2k}\,\Phi({(-1)}^{n-1},2k,1)\eta(p+q-2k)\]
		\[-\frac12 {(-n)}^{-p}\sum_{j=2}^{n}\sum_{k=1}^{p}\binom{p+q-k-1}{q-1}{(-1)}^{k+j}\,\Theta({(-1)}^{n-1},k,j-1,n)\times\]
		\[\Phi\left(-1,p+q-k,\frac{n-j+1}{n}\right).\]
	\end{enumerate}
\end{corollary}

%%%%%%%%%%%%%%%%%%%%%%%%%%%%%%%%%%%%
\begin{theorem}\label{thm of O_n}
	Let $p,q,n\in\mathbb{Z}^{+}$, $a=\pm1$, and $b={(-1)}^{p+q-1}$, with $p\neq1$ when $b=1$, and $q\neq1$ when $a=1$. Then the following identity holds:
	\[\sum_{k=1}^\infty\frac{a^k\,\mathcal{O}_{nk}^{(p)}(b)}{k^q}=\frac{a}{2}(1+{(-1)}^{q})2^{-p}\,\Phi(a,q,1)\,\Phi(b,p,\tfrac12)\]
	\[+ab^{n-1}n^q{(-2)}^{-p}\sum_{k=0}^{\lfloor\frac{q}{2}\rfloor}\binom{p+q-2k-1}{p-1}n^{-2k}\,\Phi(ab^n,2k,1)\,\Phi(b,p+q-2k,\tfrac{1}{2})\]
	\[+\frac{ab}{2} {(-2n)}^{-p}\sum_{j=1}^{n}\sum_{k=1}^{p}\binom{p+q-k-1}{q-1}b^j{(-1)}^{k}\,\Theta(ab^n,k,2j-1,2n)\times\]
	\[\Phi\left(a,p+q-k,\frac{2n-2j+1}{2n}\right),\]
	where $\mathcal{O}_{nk}^{(p)}(b)=\sum_{m=1}^{nk}\frac{b^{m-1}}{{(2m-1)}^{p}}$, $\Phi$ is the Lerch function, and $\Theta$ is defined in (\ref{theta}).
\end{theorem}
\begin{remark}
	The case $p=b=1$ is valid only when $n=1$, as the first term containing the factor $\Phi(p,b,\tfrac 12)$ cancels with the last term in the second summation.
\end{remark}
\begin{proof} Following the previous proof, we find
	\[\mathcal{O}_{nk}^{(p)}(b)=\sum_{m=1}^{nk}\frac{b^{m-1}}{(2m-1)^p}=\frac{{(-1)}^{p-1}}{(p-1)!}\int_0^1 \ln^{p-1}(x)\frac{1-(bx^2)^{nk}}{1-bx^2}\mathrm{d} x,\]
	from which it follows that
	\[\sum_{k=1}^\infty\frac{a^k\, \mathcal{O}_{nk}^{(p)}(b)}{k^q}=\sum_{k=1}^\infty\frac{a^k}{k^q}\frac{{(-1)}^{p-1}}{(p-1)!}\int_0^1 \ln^{p-1}(x)\frac{1-(bx^2)^{nk}}{1-bx^2}\mathrm{d} x\]	
	\[=\frac{{(-1)}^{p-1}}{(p-1)}\int_0^1\frac{\ln^{p-1}(x)}{1-bx^2}\sum_{k=1}^\infty\frac{a^k-(ab^nx^{2n})^k}{k^q}\mathrm{d} x\]
	\[=\frac{{(-1)}^{p-1}}{(p-1)}\int_0^1\frac{\ln^{p-1}(x)}{1-bx^2}\left(a\,\Phi(a,q,1)-\operatorname{Li}_q\left(ab^nx^{2n}\right)\right)\mathrm{d} x\]	
	\[=2^{-p}a\,\Phi(a,q,1)\,\Phi(b,p,\tfrac{1}{2})-\frac{{(-1)}^{p-1}}{(p-1)!}\int_0^1\frac{\ln^{p-1}(x)\operatorname{Li}_q\left(ab^nx^{2n}\right)}{1-bx^2}\mathrm{d} x.\]
	The integral appears in Theorem \ref{from zero to one integral2} with $(a,b)$ replaced with $(b,ab^n)$.	
\end{proof}
%%%%%%%%%%%%%%%%%%%%%%%%%%%%%%%%
\begin{corollary} Let $(a,b)=(\pm1,1)$ in Theorem \ref{Euler sums2}. Then for $p+q$ odd
	\begin{enumerate}
		\item[i)]
		\[\sum_{k=1}^\infty\frac{ O_{nk}^{(p)}}{k^q}=\frac12(1+{(-1)}^{q})\zeta(q)\lambda(p)\]
		\[-{(-2n)}^{q}\sum_{k=0}^{\lfloor\frac{q}{2}\rfloor}\binom{p+q-2k-1}{p-1}{(2n)}^{-2k}\zeta(2k)\lambda(p+q-2k)\]
		\[+\frac12 {(-2n)}^{-p}\sum_{j=1}^{n}\sum_{k=1}^{p}\binom{p+q-k-1}{q-1}{(-1)}^{k}\,\Theta(1,k,2j-1,2n)\times\]
		\[\Phi\left(1,p+q-k,\frac{2n-2j+1}{2n}\right);\]
		\item[ii)]
		\[\sum_{k=1}^\infty\frac{{(-1)}^{k}\, O_{nk}^{(p)}}{k^q}=-\frac12(1+{(-1)}^{q})\eta(q)\lambda(p)\]
		\[+{(-2n)}^{q}\sum_{k=0}^{\lfloor\frac{q}{2}\rfloor}\binom{p+q-2k-1}{p-1}{(2n)}^{-2k}\eta(2k)\lambda(p+q-2k)\]
		\[-\frac12 {(-2n)}^{-p}\sum_{j=1}^{n}\sum_{k=1}^{p}\binom{p+q-k-1}{q-1}{(-1)}^{k}\,\Theta(-1,k,2j-1,2n)\times\]
		\[\Phi\left(-1,p+q-k,\frac{2n-2j+1}{2n}\right).\]
	\end{enumerate}
\end{corollary}
%%%%%%%%%%%%%%%%%%
%%%%%%%%%%%%%%%%%%%%%%%%%%%%
\begin{corollary}
	Let $(a,b)=(\pm1,-1)$ in Theorem \ref{Euler sums2}. Then for $p+q$ even
	\begin{enumerate}
		\item[i)]
		\[\sum_{k=1}^\infty\frac{\overline{O}_{nk}^{(p)}}{k^q}=\frac12(1+{(-1)}^{q})\zeta(q)\beta(p)\]
		\[-{(-1)}^{n}{(-2n)}^{q}\sum_{k=0}^{\lfloor\frac{q}{2}\rfloor}\binom{p+q-2k-1}{p-1}{(2n)}^{-2k}\,\Phi({(-1)}^{n},2k,1)\beta(p+q-2k)\]
		\[-\frac12 {(-2n)}^{-p}\sum_{j=1}^{n}\sum_{k=1}^{p}\binom{p+q-k-1}{q-1}{(-1)}^{k+j}\,\Theta({(-1)}^{n},k,2j-1,2n)\times\]
		\[\Phi\left(1,p+q-k,\frac{2n-2j+1}{2n}\right);\]
		\item[ii)]
		\[\sum_{k=1}^\infty\frac{{(-1)}^{k}\,\overline{O}_{nk}^{(p)}}{k^q}=-\frac12(1+{(-1)}^{q})\eta(q)\beta(p)\]
		\[+{(-1)}^{n}{(-2n)}^{q}\sum_{k=0}^{\lfloor\frac{q}{2}\rfloor}\binom{p+q-2k-1}{p-1}{(2n)}^{-2k}\,\Phi({(-1)}^{n-1},2k,1)\beta(p+q-2k)\]
		\[+\frac12{(-2n)}^{-p}\sum_{j=1}^{n}\sum_{k=1}^{p}\binom{p+q-k-1}{q-1}{(-1)}^{k+j}\,\Theta({(-1)}^{n-1},k,2j-1,2n)\times\]
		\[\Phi\left(-1,p+q-k,\frac{2n-2j+1}{2n}\right).\]
	\end{enumerate}
\end{corollary}

\section{BBP-type series}
%%%%%%%%%%%%%%%%%%%%%%%%%%%%%%%%%%%%%%%%%%%%%%%%%%%%
\begin{theorem} \label{bbp sum1}
	Let $p,q,n\in\mathbb{Z}^{+}$, $p+q$ odd, and $a,b=\pm1$, with $p\neq1$ when $b=1$, and $q\neq1$ when $a=1$. Then the following identity holds:
	\[\sum_{k=1}^\infty {(a^n)}^{k}\, \mathcal{H}_k^{(p)}(b)\sum_{j=1}^{n}\frac{a^{j-1}}{(nk+j)^q}\]
	\[=-\frac12 a^{n-1} n^{-q}\,\Phi(a^nb,p+q,1)+\frac12(1+{(-1)}^{q})\,\Phi(a,q,1)\,\Phi(b,p,1)\]
	\[-a^{n-1}b{(-n)}^{-q}\sum_{j=0}^{\lfloor\frac{q}{2}\rfloor}\binom{p+q-2j-1}{p-1}\,\Phi(a^nb,2j,1)\,\Phi(b,p+q-2j,1)\]
	\[+a^n{(-n)}^{p}\sum_{k=0}^{\lfloor\frac{p}{2}\rfloor}\binom{p+q-2k-1}{q-1}n^{-2k}\,\Phi(a^nb,2k,1)\,\Phi(a,p+q-2k,1)\]
	\[-\frac12 {(-n)}^{-q}\sum_{j=2}^{n}\sum_{k=1}^{q}\binom{p+q-k-1}{p-1}a^{j}{(-1)}^{k}\,\Theta(a^nb,k,j-1,n)\times\]
	\[\Phi\left(b,p+q-k,\frac{n-j+1}{n}\right),\]	
	where $\mathcal{H}_k^{(p)}(b)=\sum_{m=1}^{k}\frac{b^{m-1}}{m^p}$, $\Phi$ is the Lerch function, and $\Theta$ is defined in (\ref{theta}). 
\end{theorem}
\begin{remark}
	The case $p=b=1$ is valid only when $n=1$, as the second term containing the factor $\Phi(b,p,1)$ cancels with the last term in the first summation.
\end{remark}
\begin{proof}
	Replacing $x$ with $a^nx^n$ in Lemma \ref{generating function}, we get
	\[\sum_{k=1}^\infty \mathcal{H}_k^{(p)}(b)(a^nx^n)^k=b\frac{\operatorname{Li}_p(a^nbx^n)}{1-a^nx^n},\]
	from which it follows that
	\[\frac{\operatorname{Li}_p(a^nb x^n)}{1-ax}=\frac{\operatorname{Li}_p(a^nb x^n)}{1-a^n x^n}\cdot \frac{1-a^n x^n}{1-ax}=b\sum_{k=1}^\infty \mathcal{H}_k^{(p)}(b){(ax)}^{nk}\sum_{j=1}^{n}{(ax)}^{j-1}\]
	\[=b\sum_{k=1}^\infty {(a^n)}^{k}\, \mathcal{H}_k^{(p)}(b)\sum_{j=1}^{n} a^{j-1} x^{nk+j-1}.\]
	Multiplying this by $\ln^{q-1}(x)$, then integrating for $x\in(0,1)$, we obtain
	\[\int_0^1\frac{\ln^{q-1}(x)\operatorname{Li}_{p}(a^nb x^{n})}{1-ax}\mathrm{d} x=b{(-1)}^{q-1}(q-1)!\sum_{k=1}^\infty {(a^n)}^{k}\, \mathcal{H}_k^{(p)}(b)\sum_{j=1}^{n}\frac{a^{j-1}}{(nk+j)^q}.\]
	This integral appears in Theorem \ref{from zero to one integral1} with $b$ replaced by $a^nb$, completing the proof.
\end{proof}
%%%%%%%%%%%%%%%%%%%%%%%%%%%%%%%%%%%%%%%%%%%%%
\begin{corollary} Let $(a,b)=(1,\pm1)$ in Theorem \ref{bbp sum1}. Then for $p+q$ odd
	\begin{enumerate}
		\item[i)]		 		
		\[\sum_{k=1}^\infty  H_k^{(p)}\sum_{j=1}^{n}\frac{1}{(nk+j)^q}=-\frac12 n^{-q}\zeta(p+q)+\frac12(1+{(-1)}^{q})\zeta(p)\zeta(q)\]
		\[-{(-n)}^{-q}\sum_{j=0}^{\lfloor\frac{q}{2}\rfloor}\binom{p+q-2j-1}{p-1}\zeta(2j)\zeta(p+q-2j)\]
		\[+{(-n)}^{p}\sum_{k=0}^{\lfloor\frac{p}{2}\rfloor}\binom{p+q-2k-1}{q-1}n^{-2k}\zeta(2k)\zeta(p+q-2k)\]
		\[-\frac12 {(-n)}^{-q}\sum_{j=2}^{n}\sum_{k=1}^{q}\binom{p+q-k-1}{p-1}{(-1)}^{k}\,\Theta(1,k,j-1,n)\times\]
		\[\Phi\left(1,p+q-k,\frac{n-j+1}{n}\right);\]
		\item[ii)]			
		\[\sum_{k=1}^\infty  \overline{H}_k^{(p)}\sum_{j=1}^{n}\frac{1}{(nk+j)^q}=-\frac12  n^{-q}\eta(p+q)+\frac12(1+{(-1)}^{q})\eta(p)\zeta(q)\]
		\[+\frac{1}{2} {(-n)}^{-q}\sum_{j=0}^{\lfloor\frac{q}{2}\rfloor}\binom{p+q-2j-1}{p-1}\eta(2j)\eta(p+q-2j)\]
		\[+{(-n)}^{-p}\sum_{k=0}^{\lfloor\frac{p}{2}\rfloor}\binom{p+q-2k-1}{q-1}n^{-2k}\eta(2k)\zeta(p+q-2k)\]
		\[-\frac12 {(-n)}^{-q}\sum_{j=2}^{n}\sum_{k=1}^{q}\binom{p+q-k-1}{p-1}{(-1)}^{k}\,\Theta(-1,k,j-1,n)\times\]
		\[\Phi\left(-1,p+q-k,\frac{n-j+1}{n}\right).\]	
	\end{enumerate}
\end{corollary}
%%%%%%%%%%%%%%%%%%%%%%%%%%%%%%%%%%%%%%%%%%%%%
\begin{corollary} Let $(a,b)=(-1,\pm1)$ in Theorem \ref{bbp sum1}. Then for $p+q$ odd 	
	\begin{enumerate}
		\item[i)]
		\[\sum_{k=1}^\infty {({(-1)}^{n})}^{k}\, H_k^{(p)}\sum_{j=1}^{n}\frac{{(-1)}^{j-1}}{(nk+j)^q}\]
		\[=\frac12 {(-1)}^{n}n^{-q}\,\Phi({(-1)}^{n},p+q,1)+\frac12(1+{(-1)}^{q})\zeta(p)\eta(q)\]
		\[+{(-1)}^{n}{(-n)}^{-q}\sum_{j=0}^{\lfloor\frac{q}{2}\rfloor}\binom{p+q-2j-1}{p-1}\,\Phi({(-1)}^{n},2j,1)\zeta(p+q-2j)\]
		\[+{(-1)}^{n}{(-n)}^{p}\sum_{k=0}^{\lfloor\frac{p}{2}\rfloor}\binom{p+q-2k-1}{q-1}n^{-2k}\,\Phi({(-1)}^{n},2k,1)\eta(p+q-2k)\]
		\[-\frac12 {(-n)}^{-q}\sum_{j=2}^{n}\sum_{k=1}^{q}\binom{p+q-k-1}{p-1}{(-1)}^{k+j}\,\Theta({(-1)}^{n},k,j-1,n)\times\]
		\[\Phi\left(1,p+q-k,\frac{n-j+1}{n}\right);\]
		\item[ii)]			
		\[\sum_{k=1}^\infty {({(-1)}^{n})}^{k}\, \overline{H}_k^{(p)}\sum_{j=1}^{n}\frac{{(-1)}^{j-1}}{(nk+j)^q}\]
		\[=\frac12 {(-1)}^{n} n^{-q}\,\Phi({(-1)}^{n-1},p+q,1)+\frac12(1+{(-1)}^{q})\eta(p)\eta(q)\]
		\[-{(-1)}^{n}{(-n)}^{-q} \sum_{j=0}^{\lfloor\frac{q-1}{2}\rfloor}\binom{p+q-2j-1}{p-1}\,\Phi({(-1)}^{n-1},2j,1)\eta(p+q-2j)\]
		\[+{(-1)}^{n}{(-n)}^{p}\sum_{k=0}^{\lfloor\frac{p}{2}\rfloor}\binom{p+q-2k-1}{q-1}n^{-2k}\,\Phi({(-1)}^{n-1},2k,1)\eta(p+q-2k)\]
		\[-\frac12{(-n)}^{-q}\sum_{j=2}^{n}\sum_{k=1}^{q}\binom{p+q-k-1}{p-1}{(-1)}^{k+j}\,\Theta({(-1)}^{n-1},k,j-1,n)\times\]
		\[\Phi\left(-1,p+q-k,\frac{n-j+1}{n}\right).\]
	\end{enumerate}			
\end{corollary}
%%%%%%%%%%%%%%%%%%%%%%%%%%%%%%%%%%%%%%%%%%%%%%%%%%%%%%%%%%%%

%%%%%%%%%%%%%%%%%%%%%%%%%%%%%%%%%%%%%%%%%%%%%%%%
\begin{theorem}\label{bbp sum2}
	Let $p,q,n\in\mathbb{Z}^{+}$, $a={(-1)}^{p+q-1}$, and $b=\pm1$, with $p\ne1$ when $b=1$, and $q\neq1$ when $a=1$. Then the following identity holds:
	\[\sum_{k=1}^\infty {(a^n)}^{k} \mathcal{H}_k^{(p)}(b)\sum_{j=1}^{n}\frac{a^{j-1}}{(2nk+2j-1)^q}= 2^{-q-1}(1-{(-1)}^{p})\,\Phi(a,q,\tfrac12)\,\Phi(b,p,1)\]
	\[-a^{n-1} (-2)^{-q} n^{p}\sum_{k=0}^{\lfloor\frac{p}{2}\rfloor}\binom{p+q-2k-1}{q-1}n^{-2k}\,\Phi(a^nb,2k,1)\,\Phi(a,p+q-2k,\tfrac{1}{2})\]
	\[-\frac12 a {(-2n)}^{-q}\sum_{j=1}^{n}\sum_{k=1}^{q}\binom{p+q-k-1}{p-1}a^j{(-1)}^{k}\,\Theta(a^nb,k,2j-1,2n)\times\]
	\[\Phi\left(b,p+q-k,\frac{2n-2j+1}{2n}\right),\]	
	where $\mathcal{H}_k^{(p)}(b)=\sum_{m=1}^{k}\frac{b^{m-1}}{m^p}$, $\Phi$ is the Lerch function, and $\Theta$ is defined in (\ref{theta}). 
\end{theorem}
\begin{remark}
	The case $p=b=1$ is valid only when $n=1$, as the addition of the first term containing the factor $\Phi(b,p,1)$ and the $k=q$ term in the second summation simplifies to \[-2^{p-q-1}(1-{(-1)}^{p})\Phi(a,q,\tfrac12)\eta(p).\]
\end{remark}
\begin{proof} Replacing $x$ with $a^n x^{2n}$ in Lemma \ref{generating function}, and then applying the same approach as in the previous proof, we find
	\[\int_0^1\frac{\ln^{q-1}(x)\operatorname{Li}_{p}(a^nb x^{2n})}{1-ax^2}\mathrm{d} x=b{(-1)}^{q-1}(q-1)!\sum_{k=1}^\infty {(a^n)}^{k} \mathcal{H}_k^{(p)}(b)\sum_{j=1}^{n}\frac{a^{j-1}}{(2nk+2j-1)^q}.\]
	This integral appears in Theorem \ref{from zero to one integral2} with $b$ replaced by $a^nb$, ending the proof.
\end{proof}

%%%%%%%%%%%%%%%%%%%%%%%%%%%%%%%
\begin{corollary}
	Let $(a,b)=(1,\pm1)$ in Theorem \ref{bbp sum2}. Then for $p+q$ odd
	\begin{enumerate}
		\item[i)]
		\[\sum_{k=1}^\infty  H_k^{(p)}\sum_{j=1}^{n}\frac{1}{(2nk+2j-1)^q}=\frac12(1-{(-1)}^{p})\lambda(q)\zeta(p)\]
		\[+{(-2n)}^{p}\sum_{k=0}^{\lfloor\frac{p}{2}\rfloor}\binom{p+q-2k-1}{q-1}{(2n)}^{-2k}\zeta(2k)\lambda(p+q-2k)\]
		\[-\frac12 {(-2n)}^{-q}\sum_{j=1}^{n}\sum_{k=1}^{q}\binom{p+q-k-1}{p-1}{(-1)}^{k}\,\Theta(1,k,2j-1,2n)\times\]
		\[\Phi\left(1,p+q-k,\frac{2n-2j+1}{2n}\right);\]
		\item[ii)]
		\[\sum_{k=1}^\infty  \overline{H}_k^{(p)}\sum_{j=1}^{n}\frac{1}{(2nk+2j-1)^q}=\frac12(1-{(-1)}^{p})\lambda(q)\eta(p)\]
		\[+{(-2n)}^{p}\sum_{k=0}^{\lfloor\frac{p}{2}\rfloor}\binom{p+q-2k-1}{q-1}{(2n)}^{-2k}\eta(2k)\lambda(p+q-2k)\]
		\[-\frac12 {(-2n)}^{-q}\sum_{j=1}^{n}\sum_{k=1}^{q}\binom{p+q-k-1}{p-1}{(-1)}^{k}\,\Theta(-1,k,2j-1,2n)\times\]
		\[\Phi\left(-1,p+q-k,\frac{2n-2j+1}{2n}\right).\]
	\end{enumerate}
\end{corollary}

%%%%%%%%%%%%%%%%%%%%%%%%%%%%%%%%%%%%%%%%%%%%%%%%%%
\begin{corollary}
	Let $(a,b)=(-1,\pm1)$ in Theorem \ref{bbp sum2}. Then for $p+q$ even
	\begin{enumerate}
		\item[i)]
		\[\sum_{k=1}^\infty ({(-1)}^{n})^{k}\, H_k^{(p)}\sum_{j=1}^{n}\frac{{(-1)}^{j-1}}{(2nk+2j-1)^q}=\frac12(1-{(-1)}^{p})\beta(q)\zeta(p)\]
		\[+{(-1)}^{n}{(-2n)}^{p}\sum_{k=0}^{\lfloor\frac{p}{2}\rfloor}\binom{p+q-2k-1}{q-1}{(2n)}^{-2k}\,\Phi({(-1)}^{n},2k,1)\beta(p+q-2k)\]
		\[+\frac12 {(-2n)}^{-q}\sum_{j=1}^{n}\sum_{k=1}^{q}\binom{p+q-k-1}{p-1}{(-1)}^{k+j}\,\Theta({(-1)}^{n},k,2j-1,2n)\times\]
		\[\Phi\left(1,p+q-k,\frac{2n-2j+1}{2n}\right);\]
		\item[ii)]
		\[\sum_{k=1}^\infty ({(-1)}^{n})^{k}\, \overline{H}_k^{(p)}\sum_{j=1}^{n}\frac{{(-1)}^{j-1}}{(2nk+2j-1)^q}=\frac12(1-{(-1)}^{p})\beta(q)\eta(p)\]
		\[+{(-1)}^{n}{(-2n)}^{p}\sum_{k=0}^{\lfloor\frac{p}{2}\rfloor}\binom{p+q-2k-1}{q-1}{(2n)}^{-2k}\,\Phi({(-1)}^{n-1},2k,1)\beta(p+q-2k)\]
		\[+\frac12{(-2n)}^{-q}\sum_{j=1}^{n}\sum_{k=1}^{q}\binom{p+q-k-1}{p-1}{(-1)}^{k+j}\,\Theta({(-1)}^{n-1},k,2j-1,2n)\times\]
		\[\Phi\left(-1,p+q-k,\frac{2n-2j+1}{2n}\right).\]
	\end{enumerate}
\end{corollary}

%%%%%%%%%%%%%%%%%%%%%%%%%%%%%%%%%
\begin{remark} To find a closed form for any combination of 
	$(p,q,a,b,n)$ from a given formula, we employ the following \textit{Mathematica} code:
	\begin{verbatim}
		Unprotect[LerchPhi];
		LerchPhi[1, 0, 1] = -1/2;
		Protect[LerchPhi];
		Theta[c_, q_, s_, r_] := 
		LerchPhi[c, q, s/r] + c (-1)^q LerchPhi[c, q, (r - s)/r];
		Theta[1, 1, s_, r_] := Pi Cot[Pi s/r]; 
		f[p_, q_, n_, a_, b_] := formula expression /. 
		Zeta[s_, r_] -> (-1)^s PolyGamma[s - 1, r]/(s - 1)!;\end{verbatim}
	The last line of the code is applied to convert the Hurwitz zeta function into the polygamma function, if the reader finds this conversion more suitable.
\end{remark}

%%%%%%%% Bibliography %%%%%%%%%%%%%%%%%%%%%%%%%%%%%%%%%%%%%%%%%%%%%%%%%

\end{document}